\documentclass[12pt]{article}

\usepackage{amsmath, amsfonts, amssymb}
\usepackage{mathrsfs}
\usepackage{theorem}
\usepackage{bm}
\RequirePackage[colorlinks,citecolor=blue,urlcolor=blue,bookmarksopen]{hyperref}

\newtheorem{mythm}{Theorem}[section]

\newtheorem{mylem}[mythm]{Lemma}

{\theorembodyfont{\rmfamily} \newtheorem{rem}[mythm]{Remark}}
{\theorembodyfont{\rmfamily} }

\newcommand{\ra}{\rightarrow}
\newcommand{\dis}{\displaystyle}

\def\R{\mathbb R}

\def\d{\text{\rm{d}}}
\def\E{\mathbb E}
\def\p{\mathbb P}\def\e{\text{\rm{e}}}
\def\q{\mathbb Q}
\def\la{\langle}
\def\raa{\rangle}
\def\La{\Lambda}
\def\veps{\varepsilon}

\def\S{\mathcal S}

\def\wt{\widetilde}

\newcommand{\fin}{\hfill $\square$\par}
\newenvironment{proof}{{\noindent\it Proof.}\ }{\hfill $\square$\par}

\numberwithin{equation}{section}

\allowdisplaybreaks

\begin{document}

\title{Stability of regime-switching processes under perturbation of transition rate matrices\footnote{Supported in
 part by NNSFs of China (Nos. 11771327, 11831014, 11431014)}}

\author{Jinghai Shao\thanks{Center for Applied Mathematics, Tianjin University, Tianjin 300072, China. Email: shaojh@tju.edu.cn.}  and Chenggui Yuan \thanks{Department of Mathematics, Swansea University, Singleton Park, SA2 8PP, UK. Email: C.Yuan@swansea.ac.uk}}

\maketitle

\begin{abstract}
  This work is concerned with the stability of regime-switching processes under the perturbation of the transition rate matrices. From the viewpoint of application, two kinds of perturbations are studied: the size of the transition rate matrix is fixed, and only the values of entries are perturbed; the values of entries and the size of the transition matrix are all perturbed. Moreover, both regular and irregular coefficients of the underlying system are investigated, which clarifies the impact of the regularity of the coefficients on the stability of the underlying system.
\end{abstract}

\noindent AMS subject Classification (2010):\ 60J27, 60J60, 60A10

\noindent \textbf{Keywords}: Regime-switching diffusions, Stability, Wasserstein distance, Transition rate matrices

\section{Introduction}
Regime-switching models have emerged in many research fields such as biological, ecological, mathematical finance, economics and storage modeling.
We refer the readers to \cite{BS, BBG, CH, GAM, PP, Sh1, SX2, XY} and the
monographs \cite{MY, YZ} for the study on ergodicity,
stochastic stability, numerical approximation of regime-switching
diffusion processes with Markovian switching or state-dependent
switching in a finite state space or in an infinite state space. These kinds of models contain two components $(X_t,\La_t)$. The first component $(X_t)$ is used to describe the dynamical system under investigation and the second component $(\La_t)$ is used to describe the random change of the environment where the dynamical system lives in. Since the impact of the change of environment has been considered in these models, they  can fit practice more precisely. Moreover, recent works have found more and more special characteristics of these models compared with those models without regime-switching. For instance, the invariant probability measures of Ornstein-Uhlenback processes and Cox-Ingersoll-Ross processes with regime-switching may be heavy tailed, whereas without regime-switching, their invariant probability measures must be light tailed; see, \cite{Bar,DY} and \cite{HS18}.

The stability of regime-switching processes is of great interest and there is a great deal of literatures in this topic; see, for example,  \cite{BBG, BBG99, M99, MY, XY, YZ} and references therein. All the aforementioned works focus on the stability of this system with respective to its equilibrium point or initial values. However, the stability of this system with respective to  the perturbation of the transition rate matrix of $(\La_t)$ has not been studied before. This kind of stability plays a crucial role in the application of the regime-switching diffusion processes; for example, performing sensitivity analysis.

In application, the random switching of the environment is observed from empirical data. Then, the transition rate matrix $(q_{ij})_{i,j\in\S}$ is estimated by statistical method based on empirical data. Therefore, the error of estimation is crucial and cannot be removed. As a consequence, the impact of this error of estimation should be evaluated.  For instance, as shown by Brown and Dybvig \cite{BD}, based on the empirical data from
US treasury yields, the poor empirical performance of the  Cox-Ingersoll-Ross model without the regime-switching well suggests the existence of regime shifts. So, one may include the regime-switching of the financial market into the Cox-Ingersoll-Ross model. It is quite possible to consider that there are three different states in the financial market: bull market, bear market and a middle market. In this case, one uses a Markov chain $(\La_t)$ in a state space $\S=\{0,1,2\}$ to characterize the random change of the financial market. There is the error of estimation for $(q_{ij})_{i,j\in\S}$ of the transition rate matrix of $(\La_t)$. On the other hand, maybe other experts would like to separate the financial market into two different states: bull market and bear market. The effects of the option pricing  by using  models with two or three states
could be quite different. Therefore, it is quite important to measure this difference.

For the regime-switching diffusions $(X_t,\La_t)$, $(X_t)$ satisfies the following stochastic differential equation (SDE for short):
\begin{equation}\label{1.1}
\d X_t=b(X_t,\La_t)\d t+\sigma(X_t,\La_t)\d W_t,\quad X_0=x_0\in\R^d,\  \La_0=i_0\in\S,
\end{equation} where $b:\R^d\times\S\ra \R^d$, $\sigma:\R^d\times \S\ra \R^{d\times d}$, $\S=\{0,1,\ldots,N\}$, $N<\infty$, and $(W_t)$ is a $d$-dimensional Brownian motion. $(\La_t)$ is a continuous-time Markov chain on $\S$ with the transition rate matrix $Q=(q_{ij})_{i,j\in\S}$. Suppose that $Q$ is conservative (i.e. $\sum_{j\in \S} q_{ij}=0$ for every $i\in\S$) and totally stable (i.e. $q_i=-q_{ii}<+\infty$ for every $i\in\S$).  Throughout this work, $(\La_t)$ and $(W_t)$ are assumed to be mutually independent.

In this work we are concerned with the stability of the process $(X_t)$ under perturbation of the transition rate matrix of $(\La_t)$.
From the application point of view, there are mainly two types of perturbations of $Q$.

\emph{First type of perturbation}: The size of $Q$ is fixed, however, each entry $q_{ij}$ of $Q$ may have small perturbation. Namely,
there is another transition rate matrix $\wt Q=(\tilde q_{ij})_{i,j\in\S}$, and each entry $\tilde q_{ij}$ acts as an estimator of the element $q_{ij}$ of $Q$. Without loss of generality, assume that $\wt Q$ is conservative and totally stable, then a unique transition function $\wt P_t,\, t\geq 0$ is determined (cf. e.g. \cite[Corollary 3.12]{Chen}). Let $(\tilde \La_t)$ be a continuous-time Markov chain starting from $i_0$ corresponding to $\wt Q$. Then the distribution of $\tilde \La_t$ is fixed,
so, a new dynamical system $(\wt X_t)$ is induced from the process $(\tilde \La_t)$, i.e.
\begin{equation}\label{1.2}
\d \wt X_t=b(\wt X_t,\tilde \La_t)\d t+\sigma(\wt X_t,\tilde \La_t)\d W(t),\quad \wt X_0=x_0\in\R^d,\ \tilde \La_0=i_0\in\S.
\end{equation}
Under some suitable conditions of the coefficients $b(\cdot,\cdot)$ and $\sigma(\cdot,\cdot)$, SDEs \eqref{1.1} and \eqref{1.2} admit a unique solution (cf. e.g. \cite{MY}). Therefore, the distributions $\mathcal{L}(X_t)$ of $X_t$ and $\mathcal{L}(\wt X_t)$ of $\wt X_t$ are determined in some sense by the transition rate matrix $Q$  and $\wt Q$ respectively. The following basic and important question therefore arises:
\begin{itemize}
\item[$-$] Can the difference between the distributions of $X_t$ and $\wt X_t$ be estimated by the difference between $Q$ and $\wt Q$?
\end{itemize}

\emph{Second type of perturbation}:  Both the entries of $Q$ and the size of $Q$ can be changed.  In application, when facing the graphs drawn from  experimental data,  it is hard sometimes to determine the number of the regimes for  the regime-switching processes. For example, if there are actually three regimes,  the process stays for a very short period of time at one of them. From this kind of experimental data,  it is very likely that a regime-switching model with only two regimes is detected.
What is the impact caused by this incorrect choice of the number of states for the regime-switching processes?

Precisely, let $\widehat Q$ be a conservative transition rate matrix on  $E:=\S\backslash\{0,1,\ldots, m\}$ with $m<N$, which determines uniquely the semigroup $\hat P_t=\e^{t\widehat Q},\ t\geq 0$ on $E$. Let $(\hat \La_t)$ be a continuous-time Markov chain on $E$ corresponding to $(\hat P_t)$ or equivalently $\widehat Q$. Using the same coefficients $b(\cdot,\cdot)$, $\sigma(\cdot,\cdot)$ as those of SDE \eqref{1.1},
we consider a new dynamical system $(\hat X_t)$ corresponding to $(\hat \La_t)$ defined by:
\begin{equation}\label{1.3}
\d \hat X_t=b(\hat X_t,\hat \La_t)\d t+\sigma(\hat X_t,\hat \La_t)\d W_t, \quad \hat X_0=x_0\in\R^d,\ \hat \La_0=i_0\in E.
\end{equation}
Under suitable conditions of $b$ and $\sigma$, the solutions of \eqref{1.1} and \eqref{1.3} are uniquely determined (cf. \cite{MY}). This means that given $\widehat Q$ on $E$, the distribution of $\hat X_t$ is then determined. Denote $\mathcal{L}(X_t)$ and $\mathcal{L}(\hat X_t)$ the distributions of $X_t$ and $\hat X_t$ respectively.
We aim to measure the Wasserstein distance $W_2(\mathcal{L}(X_t),\mathcal{L}(\hat X_t))$ via the difference between the transition rate matrices $Q=(q_{ij})_{i,j\in\S}$ and $\hat Q=(\hat q_{ij})_{i,j\in E}$. To achieve this, reformulate $Q$ into the following form:
\begin{equation}\label{1.4}
Q=\begin{pmatrix}
  Q_0& A\\
  B & Q_1
\end{pmatrix},
\end{equation} where $Q_0\in \R^{m\times m}$, $A\in \R^{m\times (N-m)}$,
$B\in\R^{(N-m)\times m}$, and $Q_1\in \R^{(N-m)\times (N-m)}$.

Our method in this paper establishes a connection between the stability of regime-switching processes with the perturbation theory of the continuous time Markov chains under the help of Skorokhod's representation theory for Markov chains. This result develops the classical perturbation theory (cf. e.g. \cite{Mit03,Mit04,ZI}) focusing on the difference of fixed time $t$ to that of a time interval $[0,t]$. The perturbation theory of continuous time Markov chain was applied to study the strong ergodicity of Markov chain (cf. \cite{ZI} and references therein), and to perform sensitivity analysis (cf. \cite{Mit03,Mit04}). In this paper, we demonstrate its connection with the stability of regime-switching processes, allowing us to performing sensitivity analysis for regime-switching processes arising from applications. In addition, to clarify the impact of the regularity of the drifts of the underlying system on this stability issue, we consider the system with regular coefficients (i.e. satisfying one-sided Lipschitz condition) and irregular coefficients (i.e. satisfying integrability condition). To deal with the irregular case, we apply a technique based on the dimension-free Harnack inequality. The coefficients in the irregular case can be very singular; see example \eqref{c-2} below.

Let us first consider the situation that the coefficients of \eqref{1.1} are regular. Assume the coefficients $b:\R^d\times\S\ra \R^d$ and $\sigma:\R^d\times\S\ra \R^{d\times d}$ satisfy:
{ \begin{itemize}
  \item[$\mathrm{(H1)}$] For each $i\in\S$ there exists a constant $\kappa_i$ such that
  \[2\la x-y,b(x,i)-b(y,i)\raa+2\|\sigma(x,i)-\sigma(y,i)\|_{\mathrm{HS}}^2\leq \kappa_i|x-y|^2,\quad x,\,y\in \R^d.\]
  \item[$\mathrm{(H2)}$] There exists a constant $K$ such that
  \[|b(x,i)|^2\leq K(1+|x|^2),\quad \|\sigma(x,i)\|_{\mathrm{HS}}^2\leq K(1+|x|^2),\quad x\in \R^d,\ i\in\S.\]
\end{itemize} }
In this case, we shall use the Wasserstein distance $W_2(\cdot,\cdot)$ to measure the difference between the distributions of $X_t$ and $\wt X_t$, which is defined by
\begin{equation}\label{wass}
W_2(\nu_1,\nu_2)^2=\inf_{\Pi\in\mathcal{C}(\nu_1,\nu_2)}
\Big\{\int_{\R^d\times\R^d}|x-y|^2\Pi(\d x,\d y)\Big\},
\end{equation}
where $\mathcal{C}(\nu_1,\nu_2)$ denotes the set of all probability measures on $\R^d\times\R^d$ with marginals $\nu_1$ and $\nu_2$.
To measure the difference between $Q$ and $\wt Q$, we use the $\ell_1$-norm $\|Q-\wt Q\|_{\ell_1}$ (i.e. the maximum absolute row sum norm) in this work, but other norm of matrix still works.

To state our results, we first introduce some notation. For an irreducible transition rate matrix $Q$ on $\S$, its corresponding transition probability measure $P_t(i,\cdot)$ must be strongly ergodic (cf. e.g. \cite[Theorems 4.43, 4.44]{Chen}). Denote $\pi=(\pi_i)$ the invariant probability measure of $Q$. Define $\tau$ to be the largest positive constant such that
 \begin{equation}\label{b-c}
 \sup_{i\in \S} \|P_t(i, \cdot)-\pi\|_{\mathrm{var}}=O(\e^{-\tau t}),\quad t>0,
 \end{equation}
 where $\|\mu-\nu\|_{\mathrm{var}}$ stands for the total variation distance between two probability measures $\mu$ and $\nu$, i.e. $\|\mu-\nu\|_\mathrm{var}=2\sup\{|\mu(A)-\nu(A)|; A\in \mathscr B(S)\}$.
 Additionally,  for $p>0$, let
  \[Q_p=Q+p\,\mathrm{diag}(\kappa_0,\kappa_1,\ldots,\kappa_N),\]
  and
  \begin{equation}\label{eta}\eta_p=-\max\big\{\mathrm{Re}(\gamma);\,\gamma\in \mathrm{spec}(Q_p)\big\},
  \end{equation}
  where $\mathrm{diag}(\kappa_0,\kappa_1,\ldots,\kappa_N)$ denotes the diagonal matrix generated by the vector $(\kappa_0,\kappa_1,$ $\ldots,\kappa_N)$, $\mathrm{spec}(Q_p)$ denotes the spectrum of the operator $Q_p$.

We are now in the position to state our  main results of this work for SDEs with regular coefficients. The first result is about the estimate of the difference of distributions of the solutions of \eqref{1.1} and \eqref{1.2}.
\begin{mythm}\label{t1}
Let $(X_t,\La_t)$ and $(\wt X_t,\tilde \La_t)$ be the solution of \eqref{1.1} and \eqref{1.2} respectively. Assume $\mathrm{(H1)}$ and $\mathrm{(H2)}$ hold. Then
\begin{equation}\label{a-1}
\begin{split}
  W_2(\mathcal{L}(X_t),\mathcal{L}(\wt X_t))^2&\leq \big(  4\veps^{-1}\!+\!8\big)KC_2(p)^{\frac 1p}\Big(N^{ 2 }t^{  2 }\|Q-\wt Q\|_{\ell_1}\Big)^{\frac 1q}\Psi(t,\veps,\eta_p,K,p),
\end{split}
\end{equation} where $p>1$, $q=p/(p-1)$, $\veps$ and $C_2(p)$ are positive constants, $\eta_p$ is defined by \eqref{eta}, and 
\begin{equation}\label{psi}
\Psi(t,\veps,\eta_p,K,p)=\Big(\!\int_0^t\big[1+(|x_0|^2+2Ks)\e^{(2K+1)s}\big]^p\e^{-(\eta_p-\veps p)(t-s)}\d s\Big)^{\frac 1p}.
\end{equation}
If assume further that
\begin{equation}\label{a-2}
|b(x,i)|^2\leq K,\quad \|\sigma(x,i)\|_{\mathrm{HS}}^2\leq K,\quad x\in\R^d,\ i\in\S,
\end{equation}
then we have a simple estimate:
\begin{equation}\label{a-3}
\begin{split}
&W_2(\mathcal{L}(X_t),\mathcal{L}(\wt X_t))^2\\&\leq (4\veps^{-1}\!+\!8)K C_2(p)^{\frac 1p} \big(N^2 t^2 \|Q-\wt Q\|_{\ell_1}\big)^{\frac 1q}
\Big(\frac{1-\e^{- (\eta_p-\veps p) t}}{\eta_p-\veps p}\Big)^{\frac 1p}.
\end{split}
\end{equation}
\end{mythm}
The second result is about the estimate of the difference of distributions of the solutions of \eqref{1.1} and \eqref{1.3}.
\begin{mythm}\label{t2}
Let $(X_t,\La_t)$ and $(\hat X_t,\hat \La_t)$ be the solutions of \eqref{1.1} and \eqref{1.3} respectively. Suppose $\tilde \La_0=\La_0\in E$.  Assume $\mathrm{(H1)}$ and $\mathrm{(H2)}$  hold. Then
\begin{equation}\label{b-1}
\begin{split}
&W_2(\mathcal{L}(X_t),\mathcal{L}(\hat X_t))^2\\
&\leq \big(  4\veps^{-1}\!+\!8\big)KC_2(p)^{\frac 1p}\big(Nt)^{\frac 2q}\Big(\|B\|_{\ell_1}+\|Q_1-\widehat Q\|_{\ell_1}\Big)^{\frac 1q}\Psi(t,\veps,\eta_p,K,p),
\end{split}
\end{equation}  where $p>1$, $q=p/(p-1)$, $\veps$ and $C_2(p)$ are positive constants, $\eta_p$ is defined by \eqref{eta}, and $\Psi(t,\veps,\eta_p,K,p)$ is given by \eqref{psi}.
Assume further that $b$ and $\sigma$ satisfy  \eqref{a-2}, then
\begin{equation}\label{b-2}
\begin{split}
&W_2(\mathcal{L}(X_t),\mathcal{L}(\hat X_t))^2\\
&\leq (4\veps^{-1}\!+\!8)K C_2(p)^{\frac 1p} \big(N t\big)^{\frac 2q}\Big(\|B\|_{\ell_1}+ \|Q_1-\widehat Q\|_{\ell_1}\Big)^{\frac 1q}
\Big(\frac{1-\e^{- (\eta_p-\veps p) t}}{\eta_p-\veps p}\Big)^{\frac 1p}.
\end{split}
\end{equation}
\end{mythm}

Next, we consider the stability of the dynamical system $(X_t)$ under the perturbation of the transition rate matrix when the coefficients of the underlying SDE are irregular.  Precisely, let
\begin{equation}\label{c-1}
\d X_t= b(X_t,\La_t)\d t+\sigma(X_t)\d W_t,\quad X_0=x_0\in\R^d, \ \La_0=i_0\in\S,
\end{equation}
where $\sigma:\R^d\ra \R^{d\times d}$ is still Lipschitz continuous,  but $b$ only satisfies some integrability condition. Here, $(\La_t)$ is also a continuous time Markov chain with a conservative and irreducible transition rate matrix $Q=(q_{ij})_{i,j\in\S}$. $(\La_t)$ is assumed to be independent of $(W_t)$. A typical example of the irregular drift $b$ concerned in this work is
\begin{equation}\label{c-2}
b(x,i)=\beta_i\Big\{\sum_{k=1}^\infty \log\Big(1+\frac{1}{|x-k|^2}\Big)\Big\}^{\frac 12}-x,
\end{equation}
where $\beta:\S\ra \R_+$. This drift $b$ is rather singular, whereas we can show that $(X_t)$ is still stable in a suitable sense w.r.t. the perturbation of $Q$ even in this situation. There are lots of researches on SDEs with irregular drifts in the form \eqref{c-2} or in $L^p([0,\infty);L^q(\R^d))$. We refer the readers to the recent works \cite{Wa17,Zh16} and references therein for more details on the motivations and applications.

Similar  to \eqref{1.2} and \eqref{1.3}, we consider the processes $(\wt X_t)$  and $(\hat X_t)$ corresponding to the perturbations $\wt Q=(\tilde q_{ij})_{i,j\in\S}$ and $\hat{Q}=(\hat{q}_{ij})_{i,j\in E}$.
Namely,
\begin{equation}\label{g-1}
\d \wt X_t=b(\wt X_t,\tilde \La_t)\d t+\sigma(\wt X_t)\d W_t, \ \wt X_0=x_0,\ \tilde \La_0=i_0,
\end{equation} where $(\tilde \La_t)$ is associated with $\wt Q$ and is independent of $(W_t)$.
\begin{equation}\label{g-2}
  \d \hat X_t=b(\hat X_t,\hat \La_t)\d t+\sigma(\hat X_t)\d W_t, \ \hat X_0=x_0, \ \hat \La_0=i_0\in E,
\end{equation} where $(\hat \La_t)$ is associated with $\widehat Q$ on the state space $E$ and is independent of $(W_t)$.
We shall measure the difference between the distribution $\mathcal{L}(X_t)$ and $\mathcal{L}(\wt X_t)$ by the Fortet-Mourier distance (also called bounded Lipschitz distance):
\begin{equation}\label{bL-dis}
\begin{aligned}
W_{bL}(\mu,\nu)=\sup\Big\{\int_{\R^d}\phi\,\d \mu-\int_{\R^d}\phi\,\d\nu;\ \|\phi\|_{\mathrm{Lip}}+\|\phi\|_\infty\leq 1\Big\}
\end{aligned}
\end{equation} for two probability measures  $\mu,\,\nu$  on $\R^d$, $\|\phi\|_{\mathrm{Lip}}:=\sup_{x,y,\in \R^d,x\neq y}\frac{|\phi(x)-\phi(y)|}{|x-y|}$. The Fortet-Mourier distance can also characterize the weak convergence of the probability measure space (cf. \cite[Chapter 6]{Vi09a}), and it is closely related to the $L_1$-Wasserstein distance via the Kantorovich-Rubinstein Theorem
(cf. \cite[Theorem 1.14]{Vi}).

To provide a suitable integrability condition on the drift $b$, we need to introduce an auxiliary function $V$ and its associated probability measure $\mu_0$. Let $V\in C^2(\R^d)$, define
\begin{equation}\label{z0}
Z_0(x)=-\sum_{i,j=1}^d\big(a_{ij}(x)\partial_j V(x)\big)e_i,
\end{equation}
where $(a_{ij}(x))=\sigma(x)\sigma^{\ast}(x)$, $\sigma^\ast$ denotes the transpose of $\sigma$ given in \eqref{c-1}, $\{e_i\}_{i=1}^d$ is the canonical orthonormal basis of $\R^d$ and $\partial_j$ is the directional derivative along $e_j$.
Let
\begin{equation}\label{mu0}
\mu_0(\d x)=\e^{-V(x)}\d x.
\end{equation}
Assume that $V$ satisfies:
\begin{itemize}
  \item[($\mathrm{A})$] there exists a $K_0>0$ such that $|Z_0(x)-Z_0(y)|\leq K_0|x-y|$ for all $x,\,y\in\R^d$, and $\mu_0(\R^d)=1$.
\end{itemize}
Let
\begin{equation}\label{z}
 Z(x,i)=b(x,i)-Z_0(x),\quad x\in\R^d,\ i\in\S.
\end{equation}

For the example $b$ in \eqref{c-2}, we can take $V(x)=x^2/2+\log\sqrt{2\pi}$, then
$Z_0(x)=-x$ and $\mu_0(\d x)=\frac{\e^{-x^2/2}}{\sqrt{2\pi}}\d x$. Also, the integrability condition \eqref{con-1} below can be verified by direct calculation for this example. In this part, for $f\in \mathscr{B}(\R^d)$, $\mu_0(f)$ denotes $\int_{\R^d} f(x)\mu_0(\d x)$.
\begin{mythm}\label{t3}
Let $(X_t,\La_t)$ be a solution of \eqref{c-1} and $(\wt X_t,\tilde \La_t)$ a solution of \eqref{g-1}.
Suppose  $V\in C^2(\R^d)$ satisfying condition $\mathrm{(A)}$.
Let $T>0$ be fixed.
Assume that there exists a constant $\eta>2T d$ such that
\begin{equation}\label{con-1}
 \max_{i\in\S}\mu_0\Big(\e^{\eta |\sigma^{-1}(\cdot)Z(\cdot,i)|^2}\Big)<\infty.
\end{equation}
Then
\begin{equation}\label{1.5}
\begin{split}
&W_{bL}(\mathcal{L}(X_t),\mathcal{L}(\wt X_t))\leq C \max\Big\{\|Q-\wt Q\|_{\ell_1}^{\frac 1{2q_0}},\|Q-\wt Q\|_{\ell_1}^{\frac{1}{2q_0\gamma}}\Big\},\quad t\in[0,T],
\end{split}
\end{equation} for some constant $C$ depending on $N, T, x_0, \tau_1, K_0, \gamma, p_0$ and $\max_{i\in\S}\mu_0\Big(\e^{\eta |\sigma^{-1}(\cdot)Z(\cdot,i)|^2}\Big)$, where  $p_0>1$ is a constant satisfying $2p_0^2 Td<\eta$, $q_0=p_0/(p_0-1)$ and  $\gamma>1$ is a constant.
\end{mythm}

\begin{mythm}\label{t4}
Let $(X_t,\La_t)$ be a solution of \eqref{c-1} and $(\hat X_t,\hat \La_t)$ a solution of \eqref{g-2}. Suppose $V\in C^2(\R^d)$ satisfying condition (A).
Let $T>0$ be fixed.
Assume there exists a constant $\eta>2T d$ such that \eqref{con-1} holds. Suppose \eqref{1.4} holds.
Then
\begin{equation}\label{1.6}
\begin{split}
&W_{bL}(\mathcal{L}(X_t),\mathcal{L}(\hat{X}_t))\\
&\leq C \max\Big\{\big(\|B\|_{\ell_1}+\|Q_1-\widehat Q\|_{\ell_1}\big)^{\frac 1{2q_0}}, \big(\|B\|_{\ell_1}+\|Q_1-\widehat Q\|_{\ell_1}\big)^{\frac{1}{2q_0\gamma}}\Big\},\quad t\in[0,T],
\end{split}
\end{equation} for some constant $C$ depending on $N, T, x_0, \tau_1, K_0, \gamma, p_0$ and $\max_{i\in\S}\mu_0\Big(\e^{\eta |\sigma^{-1}(\cdot)Z(\cdot,i)|^2}\Big)$, where   $p_0>1$ is a constant satisfying $2p_0^2 Td<\eta$, $q_0=p_0/(p_0-1)$ and  $\gamma>1$ is a constant.
\end{mythm}

\section{Proofs of  main results}

\subsection{SDEs with regular coefficients}

Let us first introduce the probability space $(\Omega,\mathscr F, \p)$ used throughout this work.
Let
\[\Omega_1=\big\{\omega\big|\,\omega:[0,\infty)\ra \R^d\ \text{continuous},\ \omega_0=0\big\}, \]
which is endowed with the local uniform convergence topology and the Wiener measure $\p_1$ so that its coordinate process $W(t,\omega)=\omega(t)$, $t\geq 0$, is  a $d$-dimensional Brownian motion.
Put
\[\Omega_2=\big\{\omega\big|\,\omega:[0,\infty)\ra \S\ \text{right continuous with left limits}\big\},\]
endowed with the Skorokhod topology and a probability measure $\p_2$. The Markov chains $(\La_t)$ and $(\tilde \La_t)$ are all constructed in the space $(\Omega_2,\mathscr B(\Omega_2),\p_2)$. Set
\[(\Omega,\mathscr F,\p)=(\Omega_1\times\Omega_2,\mathscr B(\Omega_1)\times\mathscr B(\Omega_2),\p_1\times\p_2).\]
Thus under $\p=\p_1\times\p_2$, $(\La_t)$ and $(\tilde \La_t)$ are independent of the Brownian motion $(W_t)$. Denote by $\E_{\p_1}$ taking the expectation with respect to the probability measure $\p_1$, and similarly $\E_{\p_2}$.

Next, we construct a coupling process $(\La_t, \tilde \La_t)$ such that $(\La_t)$ and $(\tilde \La_t)$ are continuous-time Markov chains with transition rate matrix $Q$ and $\wt Q$ respectively.
Denote $H=\max_{i\in\S}\{q_i,\tilde q_i\}$ and $M=N(N-1)H$.
Let $\xi_k$, $k=1,2,\ldots$, be random variables supported on $[0,M]$ satisfying
$\p_2(\xi_k\in \d x)=\mathbf{m}(\d x)/M$ where $\mathbf{m}(\d x)$ stands for the Lebesgue measure on $[0,M]$. Let $\tau_k$, $k=1,2,\ldots,$ be nonnegative random variables such that $\p_2(\tau_k>t)=\exp(-M t)$, $t\geq 0$. Suppose that
$\{\xi_k\}$ and $\{\tau_k\}$ are mutually independent. Let
\[\zeta_1=\tau_1, \zeta_2=\tau_1+\tau_2,\ldots,\zeta_k=\tau_1+\tau_2+\ldots+\tau_k, \quad k\geq 1,\]
and
\[\mathcal D_{p_1}=\{\zeta_1,\zeta_2,\ldots,\zeta_k,\ldots\}.\]
After constructing such random variables, define
\[p_1(\zeta_k)=\xi_k,\quad k\geq 1,\]
and further define the Poisson random measure
\[N_1((0,t]\times U)=\#\big\{s\in \mathcal D_{p_1};s\leq t,\,p_1(s)\in U\big\}, \ t>0, \,U\in \mathscr B(\R).\]

Construct two families of left-closed, right-open intervals $\{\Gamma_{ij}\}_{i,j\in \S}$ and $\{\wt\Gamma_{ij}\}_{i,j\in\S}$ on the half line in the following manner:
\begin{align*}
  \Gamma_{12}&=[0,q_{12}),\qquad\qquad\quad\  \wt \Gamma_{12}=[0,\tilde q_{12}),\\
  \Gamma_{13}&=[q_{12},q_{12}+q_{13}),\qquad \, \wt \Gamma_{13}=[\tilde q_{12},\tilde q_{12}+\tilde q_{13}),\\
  &\ldots\ldots\\
  \Gamma_{21}&=[q_1,q_1+q_{21}),\qquad\quad \wt \Gamma_{21}=[\tilde q_1,\tilde q_1+\tilde q_{21}),
\end{align*}and so on. For convenience of notation, put $\Gamma_{ii}=\wt \Gamma_{ii}=\emptyset$, and $\Gamma_{ij}=\emptyset$ if $q_{ij}=0$; $\wt \Gamma_{ij}=\emptyset$ if $\tilde q_{ij}=0$. Define functions $h,\,\tilde h:\S\times\R\ra \R$ by
\begin{align*}
  h(i,z)&=\sum_{\ell\in\S}(\ell-i)\mathbf 1_{\Gamma_{i\ell}}(z),\\
  \tilde h(i,z)&=\sum_{\ell\in \S}(\ell-i)\mathbf 1_{\wt \Gamma_{i\ell}}(z).
\end{align*}
Then, according to \cite[Chapter II]{Sk89} or \cite{YZ}, the solution of the SDE
\begin{equation}\label{lam-1}
\d \La_t=\int_{[0,M]}h(\La_{t-}, z)N_1(\d t,\d z),\quad \La_0=i_0,
\end{equation}
is a continuous-time Markov chain with transition rate matrix $Q=(q_{ij})$.
Similarly, the solution of the SDE
\begin{equation}\label{lam-2}
\d \tilde \La_t=\int_{[0,M]}\tilde h(\tilde \La_{t-},z)N_1(\d t,\d z),\quad \tilde \La_0=i_0,
\end{equation} is a continuous-time Markov chain with transition rate matrix $\wt Q=(\tilde q_{ij})$.
Therefore, through the SDEs \eqref{lam-1} and \eqref{lam-2}, we construct the desired coupling process $(\La_t,\tilde \La_t)$. Furthermore,
consider the following SDEs:
\begin{equation}\label{2.1}
\d X_t=b(X_t,\La_t)\d t+\sigma(X_t,\La_t)\d W_t,\quad X_0=x_0,\ \La_0=i_0,
\end{equation}
\begin{equation}\label{2.2}
\d \wt X_t=b(\wt X_t,\tilde \La_t)\d t+\sigma(\wt X_t,\tilde \La_t)\d W_t,\quad \wt X_0=x_0,\ \tilde \La_0=i_0.
\end{equation}
Then, the system $(X_t,\La_t)$ given by \eqref{2.1} and \eqref{lam-1} has the same distribution as the system given in \eqref{1.1}. Similarly, $(\wt X_t,\tilde \La_t)$ given by \eqref{2.2} and \eqref{lam-2} has the same distribution as the system given in  \eqref{1.2}. Under the help of the constructed systems $(X_t,\La_t)$ and $(\wt X_t,\tilde \La_t)$ in this section, we can provide the proof of Theorem \ref{t1}.

\begin{mylem}\label{lem-1} Let $(X_t,\La_t)$, $(\wt X_t,\tilde \La_t)$ be the solution of \eqref{2.1} and \eqref{2.2} respectively with $X_0=\wt X_0=x_0\in\R^d$.
Assume $\mathrm{(H2)}$ holds. Then, for $\p_2$-almost surely $\omega_2\in \Omega_2$,
\begin{equation}\label{elem-1}
\begin{split}
  \E_{\p_1}[|X_t|^2](\omega_2)&\leq (|x_0|^2+2Kt)\e^{(2K+1)t},\\
  \E_{\p_1}[|\wt X_t|^2](\omega_2)&\leq (|x_0|^2+2Kt)\e^{(2K+1)t},\quad t>0.
\end{split}
\end{equation}
\end{mylem}

\begin{proof}
  By It\^o's formula and (H2),
  \begin{align*}
    \d |X_t|^2&=\big[2\la X_t,b(X_t,\La_t)\raa +\|\sigma(X_t,\La_t)\|_{\mathrm{HS}}^2\big]\d t+2\la X_t,\sigma(X_t,\La_t)\d W_t\raa\\
    &\leq \big[|X_t|^2+2K(1+|X_t|^2)\big]\d t+2\la X_t,\sigma(X_t,\La_t)\d W_t\raa.
  \end{align*}
  Taking the expectation w.r.t. $\p_1$ and using   Gronwall's inequality, we obtain
  \begin{equation*}
    \E_{\p_1}[|X_t|^2](\omega_2)\leq (|x_0|^2+2Kt)\e^{(2K+1)t},\quad \text{$\p_2$-a.s.}\  \omega_2.
  \end{equation*}
  Similarly, the estimate on $\E_{\p_1}[|\wt X_t|^2](\omega_2)$ holds.
\end{proof}

\begin{mylem}\label{lem-2}
For the processes $(\La_t)$ and $(\tilde \La_t)$ given in \eqref{lam-1} and \eqref{lam-2} respectively, it holds
\begin{equation}\label{elem}
\int_0^t\p(\La_s\neq \tilde \La_s)\d s\leq N^2 t^2\|Q-\wt Q\|_{\ell_1}.
\end{equation}
\end{mylem}

\begin{proof}
  Let $\Gamma_{ij}\Delta\wt \Gamma_{ij}=\big(\Gamma_{ij}\backslash \wt\Gamma_{ij}\big)\bigcup \big(\wt\Gamma_{ij}\backslash \Gamma_{ij}\big)$. By virtue of the construction of $\Gamma_{ij}$ and $\wt \Gamma_{ij}$, we have
  \begin{align*}
    \mathbf{m}(\Gamma_{ij}\Delta\wt\Gamma_{ij})&\leq \Big|\sum_{k=1}^{i-1} q_k+\sum_{k=1,k\neq i}^{j-1} q_{ik}-\sum_{k=1}^{i-1} \tilde q_k -\sum_{k=1,k\neq i}^{j-1}\tilde q_{ik}\Big|\\
    &\quad+\Big|\sum_{k=1}^{i-1}q_k+\sum_{k=1,k\neq i}^jq_{ik}-\sum_{k=1}^{i-1}\tilde q_k-\sum_{k=1,k\neq i}^{j} \tilde q_{ik}\Big|\\
    &\leq 2(i-1)\|Q-\wt Q\|_{\ell_1}+\|Q-\wt Q\|_{\ell_1}\\
    &\leq 2N\|Q-\wt Q\|_{\ell_1}.
  \end{align*}
  See also \cite{Sh15} for more details on previous calculation.

   For $\delta\in (0,1)$ and $s>0$, let $s_\delta=[\frac{s}{\delta}]$, the integer part of $s/\delta$. Let $N(t)=N_1((0,t]\times \R)$. For every $t\in (0,\delta]$, since $\La_0=\tilde \La_0=i_0$, we have
   \begin{align*}
     \p(\La_t\neq \tilde \La_t)&=\p(\La_t\neq \tilde \La_t, N(t)\geq 1))\\
     &=\p(\La_t\neq \tilde \La_t,N(t)=1)+\p(\La_t\neq \tilde \La_t,N(t)\geq 2).
   \end{align*}
   There is a constant $C>0$ such that
   \begin{equation}\label{l-1}
   \p(N(t)\geq 2)\leq \p(N(\delta)\geq 2)=1-\e^{-M\delta}-M\delta \e^{-M\delta}\leq C\delta^2.
   \end{equation}
   On the other hand,
   \begin{align*}
     \p(\La_t\neq \tilde \La_t,N(t)=1)&=\int_0^t\p(\La_t\neq \tilde \La_t,\tau_1\in \d s,\,\tau_2>t-s)\\
     &=\int_0^t\p\Big(\xi_1\not\in\bigcup_{j\in\S}\big(\Gamma_{i_0j}\bigcap\wt \Gamma_{i_0j}\big),\tau_1\in \d s\Big)\e^{-M(t-s)}\\
     &\leq 2N^2t \e^{-Mt}\|Q-\wt Q\|_{\ell_1}.
   \end{align*}
   Hence,
   \begin{equation}\label{l-2}
   \p(\La_t\neq \tilde \La_t)\leq C\delta^2+2N^2\delta\|Q-\wt Q\|_{\ell_1}, \quad 0<t\leq \delta.
   \end{equation}
   Note that the estimate is independent of the common initial value of $(\La_t)$ and $(\tilde \La_t)$.

   To proceed,
   \begin{align*}
     \p(\La_{2\delta}\neq \tilde \La_{2\delta})&=\p(\La_{2\delta}\neq \tilde \La_{2\delta},\La_\delta=\tilde \La_\delta)+\p(\La_{2\delta}\neq \tilde \La_{2\delta},\La_\delta\neq \tilde \La_\delta)\\
     &\leq \p(\La_{2\delta}\neq \tilde \La_{2\delta}\big|\La_\delta=\tilde \La_\delta)+\p(\La_\delta\neq \tilde \La_\delta).
   \end{align*}
   By the time-homogeneity of $(\La_t,\tilde \La_t)$ and the estimate \eqref{l-2}, it follows that
   \[\p(\La_{2\delta}\neq \tilde \La_{2\delta})\leq 2C\delta^2+4N^2\delta\|Q-\wt Q\|_{\ell_1}.
   \]
   Deduce inductively to yield that, for each $k\geq 2$,
   \begin{equation}\label{l-3}
   \p(\La_{k\delta}\neq \tilde \La_{k\delta})\leq kC\delta^2+2kN^2\delta\|Q-\wt Q\|_{\ell_1}.
   \end{equation}

   By virtue of \eqref{l-2} and \eqref{l-3}, we have that for $t>0$,
   \begin{align*}
     \int_0^t\p(\La_s\neq \tilde \La_s)\d s&=\int_0^t\p(\La_s\neq \tilde \La_s,\La_{s_\delta}= \tilde \La_{s_\delta})\d s+\int_0^t\p(\La_s\neq \tilde \La_s,\La_{s_\delta}\neq \tilde \La_{s_\delta})\d s\\
     &\leq \int_0^t\p(\La_s\neq \tilde \La_s\big|\La_{s_\delta}=\tilde \La_{s_\delta})\p(\La_{s_\delta}=\tilde \La_{s_\delta})\d s
     +\int_0^t\p(\La_{s_\delta}\neq \tilde \La_{s_\delta})\d s\\
     &\leq \int_0^t\p(\La_s\neq \tilde \La_s\big|\La_{s_\delta}=\tilde \La_{s_\delta})\d s+\sum_{k=1}^K\p(\La_{k\delta}\neq \tilde \La_{k\delta})\delta\\
     &\leq C\delta^2 t+2N^2\delta  t\|Q-\wt Q\|_{\ell_1}+\frac{C\delta^3}{2}K(K+1)\\
     &\qquad \quad \ +N^2K(K+1)\delta^2\|Q-\wt Q\|_{\ell_1},
   \end{align*} where $K=\big[\frac{t}{\delta}\big]+1$.
   Letting $\delta\downarrow 0$, we obtain that
   \begin{equation*}\label{l-4}
   \int_0^t\p(\La_s\neq \tilde \La_s)\d s\leq N^2 t^2\|Q-\wt Q\|_{\ell_1},
   \end{equation*} which concludes the proof.
\end{proof}

\begin{rem}
  The perturbation theory of continuous-time Markov chains has been developed in many works; see, e.g.
  \cite{Mit03, Mit04} and references therein. According to this theory, one can get appropriate estimate of the distance between two transition semigroups  by the distance between their corresponding transition rate matrices. Whereas, to control the term $\E\int_0^t\mathbf 1_{\{\La_s\neq \tilde \La_s\}}\d s$ which concerns the behavior of Markov chains during a time interval $[0,t]$ rather than a fixed time $t$,  one has to construct a suitable coupling process. One possible method is to use the optimal coupling for continuous-time Markov chains
  (cf. \cite[Chapter 5]{Chen}). But additional conditions on the generator of the coupling process are needed. However, we do not find an explicit condition in terms of the difference between $Q$ and $\wt Q$, for example, $\|Q-\wt Q\|_{\ell_1}$ used in this work at current stage. Our result shows once again the significant effect of Skorkhod's representation of continuous-time Markov chains which has been applied in \cite{Sh18} to deal with state-dependent regime-switching processes.
\end{rem}

\noindent\textbf{Proof of Theorem \ref{t1}}\
  For simplicity of notation, let $Z_t=X_t-\wt X_t$. Then, due to (H1) and (H2), It\^o's formula yields that
  \begin{align*}
    \d |Z_t|^2&=\big\{2\la Z_t,b(X_t,\La_t)-b(\wt X_t,\tilde \La_t)\raa +\|\sigma(X_t,\La_t)-\sigma(\wt X_t,\tilde \La_t)\|_{\mathrm{HS}}^2\big\}\d t+\d M_t\\
    &\leq \big\{\kappa_{\La_t} |Z_t|^2\!+\!2\la Z_t,b(\wt X_t,\La_t)\!-\!b(\wt X_t,\tilde \La_t)\raa \! +\!2\|\sigma(\wt X_t,\La_t)\!-\!\sigma(\wt X_t,\tilde \La_t)\|_{\mathrm{HS}}^2\big\}\d t\!+\!\d M_t\\
    &\leq \big\{(\kappa_{\La_t}+\veps)|Z_t|^2+\frac 1\veps \big(|b(\wt X_t,\La_t)|+|b(\wt X_t, \tilde \La_t)|\big)^2\mathbf{1}_{\{\La_t\neq \tilde \La_t\}} \\
    & \qquad+4\big(\|\sigma(\wt X_t,\La_t)\|_{\mathrm{HS}}^2+\|\sigma(\wt X_t,\tilde \La_t)\|_{\mathrm{HS}}^2\big)\mathbf{1}_{\{\La_t\neq \tilde \La_t\}}\big\}\d t+\d M_t\\
    &\leq \big\{(\kappa_{\La_t}+\veps)|Z_t|^2+\frac{4K}{\veps}(1+|\wt X_t|^2)\mathbf{1}_{\{\La_t\neq \tilde \La_t\}} +8K(1+|\wt X_t|^2)\mathbf{1}_{\{\La_t\neq \tilde \La_t\}}\big\}\d t+\d M_t
  \end{align*}
  for any $\veps>0$, where $ M_t=\int_0^t2\la Z_s,(\sigma(X_s,\La_s)-\sigma(\wt X_s,\tilde \La_s))\d W_s\raa$ for $t\geq 0$ is a martingale.
  Taking the expectation w.r.t. $\p_1$ on both sides of the previous inequality, we get
  \begin{equation}\label{v1}
  \begin{split}
     \d\,\E_{\p_1}[|Z_t|^2](\omega_2)&\leq \big(  4\veps^{-1}+8\big) K\E_{\p_1}\big[1+|\wt X_t|^2\big](\omega_2)\mathbf 1_{\{\La_t\neq \tilde\La_t\}}(\omega_2)\d t\\
    &\quad + (\kappa_{\La_t}+\veps)(\omega_2)\E_{\p_1}[|Z_t|^2](\omega_2)\d t.
  \end{split}
  \end{equation}
  To proceed, let us recall an elementary inequality.
  Let $u(t)$ be a real-valued differentiable function, $\alpha(t)$ and $\beta(t)$ real-valued integrable functions (not necessary nonnegative). If
  \[u'(t)\leq \alpha(t)+\beta(t) u(t),\]
  then \[u(t)\leq u(0)\e^{\int_0^t\beta(s)\d s}+\int_0^t\alpha(s)\e^{\int_s^t\beta(r)\d r}\d s.\]
  Using this inequality to \eqref{v1}, and invoking the estimate in Lemma \ref{lem-1}, we obtain that
  \begin{align*}
    \E_{\p_1}[|Z_t|^2](\omega_2)&\leq \big(  4\veps^{-1}\!+\!8\big)K\!\int_0^t \!\!\Big(1\!+\!(|x_0|^2\!+\!2Ks)\e^{(2K+1)s}\Big)
    \mathbf{1}_{\{\La_s\neq \tilde \La_s\}}\,
    \e^{\int_s^t (\kappa_{\La_r}\!+\veps)(\omega_2)\d r}\d s.
  \end{align*}
  Taking the expectation w.r.t. $\p_2$ and using H\"older's inequality, we get
  \begin{equation}\label{2.3}
  \begin{split}
    \E|Z_t|^2&\leq \int_0^t\Big\{(4\veps^{-1}+8)K\big[1+(|x_0|^2+2Ks)\e^{(2K+1)s}\big]\\
    &\qquad \cdot\big(\E\mathbf 1_{\{\La_s\neq \tilde \La_s\}}(\omega_2)\big)^{\frac 1q}\big(\E\e^{p\int_s^t(\kappa_{\La_r}+\veps)(\omega_2)\d r}\big)^{\frac 1p}\Big\} \d s
  \end{split}
  \end{equation}
  for $p,\,q>1$ with $1/p+1/q=1$.

  In order to estimate the term $\E\,\e^{p\int_0^t(\kappa_{\La_s}+1)\d s}$, we need the following notation.
  Let
  \[Q_p=Q+p\,\mathrm{diag}(\kappa_0,\kappa_1,\ldots,\kappa_N),\]
  and
  \[\eta_p=-\max\big\{\mathrm{Re}(\gamma);\,\gamma\in \mathrm{spec}(Q_p)\big\}.\]
  According to \cite[Proposition 4.1]{Bar}, for any $p>0$, there exist two positive constants $C_1(p)$ and $C_2(p)$ such that
  \begin{equation}\label{ee}
  C_1(p)\e^{-\eta_p t}\leq \E\,\e^{p\int_0^t \kappa_{\La_s}\d s}\leq C_2(p)\e^{-\eta_p t}, \quad t>0.
  \end{equation}


 %
  The term $\int_0^t\E\mathbf 1_{\{\La_s\neq \tilde\La_s\}}\d s$ is estimated in Lemma \ref{lem-2}.
  Consequently, substituting the estimates \eqref{ee} and \eqref{elem} into \eqref{2.3}, we get
  \begin{equation}\label{2.5}
  \begin{split}
  \E[|Z_t|^2]&\leq  \big(  4\veps^{-1}\!+\!8\big)KC_2(p)^{\frac 1p}\Big(N^{2}t^{2}\|Q-\wt Q\|_{\ell_1}\Big)^{\frac 1q}\\
  &\qquad\cdot  \Big(\!\int_0^t\big[1+(|x_0|^2+2Ks)\e^{(2K+1)s}\big]^p\e^{-(\eta_p-\veps p)(t-s)}\d s\Big)^{\frac 1p}.
  \end{split}
  \end{equation}

  Note that the solutions of \eqref{2.1} and \eqref{2.2} exist uniquely. Then the distribution of $(X_t,\wt X_t)$ on $\R^d\times\R^d$ is a coupling of $\mathcal{L}(X_t)$ and $\mathcal{L}(\wt X_t)$. By the definition of the Wasserstein distance, it follows
  \begin{align*}
    W_2(\mathcal{L}(X_t),\mathcal{L}(\wt X_t))^2&\leq \E[|X_t-\wt X_t|^2]\\
    &\leq \big(  4\veps^{-1}\!+\!8\big)KC_2(p)^{\frac 1p}N^{\frac{2}{q}}t^{\frac 2q}\|Q-\wt Q\|_{\ell_1}^{\frac 1q}\\
  &\qquad\cdot  \Big(\!\int_0^t\big[1+(|x_0|^2+2Ks)\e^{(2K+1)s}\big]^p\e^{-(\eta_p-\veps p)(t-s)}\d s\Big)^{\frac 1p},
  \end{align*}
  which is the desired estimate \eqref{a-1}.

  When $b$ and $\sigma$ are bounded satisfying \eqref{a-2}, we have a simple estimate
  \begin{align*}
    \d |Z_t|^2&\leq \big\{  (\kappa_{\La_t}+\veps)|Z_t|^2+4K(2+\veps^{-1})\mathbf 1_{\{\La_t\neq \tilde \La_t\}}\big\}\d t+\d M_t,
  \end{align*}
  where $M_t=\int_0^t2\la Z_s,(\sigma(X_s,\La_s)-\sigma(\wt X_s,\tilde \La_s))\d W_s\raa$, $t\geq 0$. This yields
  \begin{align*}
    \E|Z_t|^2&\leq (4\veps^{-1}+8)K\Big(\int_0^t\p(\La_s\neq \tilde \La_s)\d s\Big)^{\frac 1q}\Big(\int_0^t\E\e^{p\int_s^t(\kappa_{\La_r}+\veps)\d r}\d s\Big)^{\frac 1p}.
  \end{align*}
  Then,  \eqref{a-3} can be established by  following the same procedure to deduce \eqref{a-1}. \fin

%
%

\noindent\textbf{Proof of Theorem \ref{t2}}\ To emphasize the idea, we give out the proof in the situation $E=\S\backslash\{0\}$.
For the given transition rate matrices $Q=(q_{ij})_{i,j\in\S}$ on $\S$ and $\widehat Q=(\hat q_{ij})_{i,j\in E}$ on $E$,  write $Q$ in the form
\begin{equation}\label{2.6}
Q=\begin{pmatrix}
  -q_0 & \alpha\\
  \beta & Q_1
\end{pmatrix},
\end{equation} where $\alpha=\{q_{0i}; 1\leq i\leq N\}$ and $\beta=\{q_{j0}; 1\leq j\leq N\}$ are the row and column vectors on $E$. Let $(\La_t)$ and $(\widehat \La_t)$ be the Markov chains on $\S$ and $E$ with the transition rate matrices $Q$ and $\hat Q$ respectively.
Consider
\begin{equation}\label{2.7}
\d \hat X_t=b(\hat X_t,\hat \La_t)\d t+\sigma(\hat X_t,\hat \La_t)\d W_t,\quad \hat X_0=x_0,\quad \hat \La_0=i_0\in E.
\end{equation}
In order to employ the method used in Theorem \ref{t1}, we propose the following extension
\begin{equation}\label{ext}
\wt Q=  \begin{pmatrix}
    -q_0& \alpha\\
    0&\hat Q
  \end{pmatrix}.
\end{equation}
It is easy to see that  $\wt Q$ is  conservative. Hence, there is a unique semigroup $(\wt P_t)_{t\geq 0}$ on $\S$ corresponding to the generator $\wt Q$. This $(\tilde \La_t)$ helps us to define another dynamical system $(\wt X_t)$ by the following SDE:
\begin{equation}\label{2.8}
\d \wt X_t=b(\wt X_t,\tilde \La_t)\d t+\sigma(\wt X_t,\tilde \La_t)\d W_t,\quad \wt X_0=x_0,\ \tilde \La_0=i_0\in E.
\end{equation}
Under the conditions (H1) and (H2), the solutions of SDEs \eqref{2.7} and \eqref{2.8} are uniquely determined. Due to the definition of $\wt Q$ in \eqref{ext}, the process $(\tilde \La_t)$ starting from $i_0\in E$ will never reach the point $0$, thus $\tilde \La_t=\hat \La_t$, $t>0$, a.s. when $\tilde \La_0=\hat \La_0=i_0\in E$. As a consequence,
\begin{equation}\label{2.9}
\wt X_t=\hat X_t, \quad t>0, \ \ a.s.
\end{equation}
Moreover, by virtue of \eqref{2.6} and \eqref{ext}, it holds
\begin{equation}\label{e-2.1}
\|Q-\wt Q\|_{\ell_1}\leq \|\beta\|_{\ell_1}+\|Q_1-\widehat Q\|_{\ell_1}.
\end{equation}
Following the procedure of the argument of Theorem \ref{t1},  inserting \eqref{e-2.1} into \eqref{2.5}, we obtain that
\begin{equation}\label{e-2.2}
\begin{split}
  \E[|X_t-\wt X_t|^2]&\leq  \big(  4\veps^{-1}\!+\!8\big)KC_2(p)^{\frac 1p}\big(Nt\big)^{\frac 2q}\Big(\|\beta\|_{\ell_1}+\|Q_1-\widehat Q\|_{\ell_1}\Big)^{\frac 1q}\\
  &\qquad\cdot  \Big(\!\int_0^t\big[1+(|x_0|^2+2Ks)\e^{(2K+1)s}\big]^p\e^{-(\eta_p-\veps p)(t-s)}\d s\Big)^{\frac 1p}.
\end{split}
\end{equation}

Due to \eqref{2.9}, it follows that $\E[|X_t-\hat X_t|^2]=\E[|X_t-\wt X_t|^2]$. According to the definition of the Wasserstein distance, and using the estimate \eqref{e-2.2}, we obtain
\begin{equation}\label{e-2.3}
\begin{split}
W_2(\mathcal{L}(X_t),\mathcal{L}(\hat X_t))^2
&\leq\big(  4\veps^{-1}\!+\!8\big)KC_2(p)^{\frac 1p}\big(Nt\big)^{\frac 2q}\Big(\|\beta\|_{\ell_1}+\|Q_1-\widehat Q\|_{\ell_1}\Big)^{\frac 1q}\\
  &\qquad\cdot  \Big(\!\int_0^t\big[1+(|x_0|^2+2Ks)\e^{(2K+1)s}\big]^p\e^{-(\eta_p-\veps p)(t-s)}\d s\Big)^{\frac 1p}.
\end{split}
\end{equation}

Analogously, if $b$ and $\sigma$ are bounded satisfying \eqref{a-2}, we  have
\begin{equation}\label{e-2.4}
\begin{split}
&W_2(\mathcal{L}(X_t),\mathcal{L}(\hat X_t))^2\leq \E[|X_t-\wt X_t|^2]\\
   &\leq (4\veps^{-1}\!+\!8)K C_2(p)^{\frac 1p} \big(N t\big)^{\frac 2q}\Big(\|\beta\|_{\ell_1}+ \|Q_1-\widehat Q\|_{\ell_1}\Big)^{\frac 1q}
\Big(\frac{1-\e^{- (\eta_p-\veps p) t}}{\eta_p-\veps p}\Big)^{\frac 1p}.
\end{split}
\end{equation}
This completes the proof in the situation   $E=\S\backslash \{0\}$. The general case can be proved in the same way, and the details are  omitted.

\subsection{SDEs with irregular coefficients}
In this part, we consider the regime-switching processes with irregular drifts. Precisely, consider
\begin{equation}\label{e-3.1}
\d X_t=b(X_t,\La_t)\d t+\sigma(X_t)\d W_t,\quad X_0=x_0,\ \La_0=i_0,
\end{equation} where $b:\R^d\times\S\ra \R^d$ and $\sigma:\R^d\ra \R^{d\times d}$. Here, we assume that the diffusion coefficient $\sigma$  satisfies the Lipschitz condition: there exists $K>0$ such that
\begin{equation}\label{sigma}
\|\sigma(x)-\sigma(y)\|_{\mathrm{HS}}^2\leq K|x-y|^2, \ \ \forall\,x,y\in\R^d.
\end{equation} However, the drift $b$ is assumed to satisfy certain integrability condition. Hence, it may be discontinuous. $(\La_t)$ is a continuous time Markov chain on $\S$ with the transition rate matrix $Q=(q_{ij})_{i,j\in\S}$.
Consider the perturbation $\wt Q=(\tilde q_{ij})_{i,j\in\S}$ of $Q$ and its associated Markov chain $(\tilde \La_t)$. Let
\begin{equation}\label{e-3.2}
\d \wt X_t=b(\wt X_t,\tilde \La_t)\d t+\sigma(\wt X_t)\d W_t,\quad \wt X_0=x_0,\ \tilde \La_0=i_0.
\end{equation}
The integrability condition of type \eqref{con-1} is raised by Wang \cite{Wa17} to study the nonexplosion of the solutions of SDEs by  using the dimension-free Harnack inequality. We will use the technique of \cite{Wa17} to analyze the stability of the regime-switching processes. Moreover, according to \cite[Theorem 2.1]{Wa17} and using the technique to construct the regime-switching processes with Markovian switching (cf. e.g. \cite{MY}), it is standard to show the existence and uniqueness of the solutions of SDEs \eqref{e-3.1} and \eqref{e-3.2}.


To proceed, we make some necessary preparations. Let  $(Y_t)$ be a process associated with the reference function $V\in C^2(\R^d)$:
\begin{equation}\label{pro-y}
\d Y_t=Z_0(Y_t)\d t+\sigma \d W(t),\quad Y_0=x_0,
\end{equation}
where the vector field $Z_0$ is defined by \eqref{z0}. Since $Z_0$ is globally Lipschitz continuous by condition (A), there is a unique nonexplosive solution to SDE \eqref{pro-y}. Via the process $(Y_t)$, a new representation for $(X_t)$ and $(\wt X_t)$ can be constructed
with the help of the Girsanov theorem, which is verified by the dimension-free Harnack inequality for $(Y_t)$ under appropriate integrability conditions.

Precisely, rewrite \eqref{pro-y} as
\begin{equation*}
  \d Y_t=b(Y_t,\La_t)\d t+\sigma(Y_t)\d W_t^{(1)},
\end{equation*}
where
\begin{equation}\label{w-1}
  W_t^{(1)}=W_t -\!\int_0^t \!\!\sigma(Y_s)^{-1}Z(Y_s,\La_s)\d s,  \ Z(y,i)= b(y,i)\!-\!Z_0(y),  t>0, y\in\!\R^d, i\in\!\S.
\end{equation}
If Novikov's condition
\begin{equation}\label{Nov-1}\E\e^{\frac 12\int_{0}^{T}|\sigma^{-1}(Y_s)Z(Y_s,\La_s)|^2\d s}<\infty
\end{equation}
holds, then
\begin{equation}\label{mea-1}
  \q:=\exp\Big(\int_0^T\la \sigma^{-1}(Y_s)Z(Y_s,\La_s),\d W_s\raa-\frac 12\int_{0}^{T}\!\!|\sigma^{-1}(Y_s)Z(Y_s,\La_s)|^2\d s\Big)\p
\end{equation}
is a new probability measure. Thus, the Girsanov theorem yields that
$(W_t^{(1)})_{t\in[0,T]}$ is a new Brownian motion under the probability measure $\q$. Note that the mutual independence between $(W_t)$ and $(\La_t)$ has been used herein. Consequently,
the uniqueness of the solution for the SDE \eqref{e-3.1} tells us that $(Y_t,\La_t)_{t\in[0,T]}$ under $\q$ has the same distribution as that of $(X_t,\La_t)_{t\in[0,T]}$ under $\p$. To be more precise, let us show that $(\La_t)$ and $(W_t^{(1)})$ are mutually independent under $\q$. For any bounded measurable functions $f$ on $\S$ and $g$ on $\R^d$, it holds
\begin{equation}\label{bg-1}
\begin{split}
  &\E_{\q}\big[f(\La_t)g(W_t^{(1)})\big]
  =\E_{\p}\Big[\frac{\d \q}{\d \p} f(\La_t)g(W_t^{(1)})\Big]\\
  &=\E_{\p_2}\Big[f(\La_t)\E_{\p_2}\Big[\E_{\p_1}\Big(\frac{\d \q}{\d \p} g(W_t^{(1)})\Big)\Big|\mathscr{F}_T^\La\Big]\Big]\\
  &=\E_{\p_2}\big[f(\La_t)\E_{\p_1}\big[g(W_t)\big]\big]=\E_{\p_2}\big[f(\La_t)\big]\E_{\p}\big[g(W_t)\big]\\
  &=\E_{\p}\big[f(\La_t)\big]\E_{\q}\big[g(W_t^{(1)})\big],
\end{split}
\end{equation}
where $\mathscr F_t^{\La}$ denotes the $\sigma$-field generated by the process $(\La_s)$ up to time $t$, and $\frac{\d \q}{\d \p}$ denotes the Radon-Nikodym derivative. Applying again the Grisanov theorem, we have $\dis \E_{\p_1}\Big(\frac{\d \q}{\d \p}\Big)=1$ and
\begin{align*}
  \E_{\q}\big[f(\La_t)\big]&=\E_{\p}\Big[f(\La_t)\frac{\d \q}{\d \p}\Big]\\
  &= \E_{\p_2}\Big[f(\La_t)\E_{\p_2}\Big[\E_{\p_1}\Big(\frac{\d \q}{\d \p}\Big)\Big|\mathscr{F}_T^\La\Big]\Big]\\
  &=\E_{\p_2}\big[f(\La_t)\big]=\E_{\p}\big[f(\La_t)\big].
\end{align*}
Combining this with the previous equality \eqref{bg-1}, we have
\[\E_{\q}\big[f(\La_t)g(W_t^{(1)})\big]=\E_{\q}\big[f(\La_t)\big]\E_{\q}\big[g(W_t^{(1)})\big],\]
and hence $(\La_t)$ and $(W_t^{(1)})$ are mutually independent.

Analogously, rewrite $(Y_t)$ as
\begin{equation*}
\d Y_t=b(Y_t,\tilde \La_t)\d t+\sigma(Y_t)\d \wt W_t,
\end{equation*}
where
\begin{equation}\label{w-2}
  \wt W_t=W_t-\int_0^t\sigma(Y_s)^{-1}Z(Y_s,\tilde \La_s)\d s.
\end{equation}
If Novikov's condition
\begin{equation}\label{Nov-2}\E\e^{\frac 12\int_{0}^{T}|\sigma^{-1}(Y_s)Z(Y_s,\tilde \La_s)|^2\d s}<\infty
\end{equation}
holds, then
\begin{equation}\label{mea-1}
  \wt\q:=\exp\Big(\int_0^T\la \sigma^{-1}(Y_s)Z(Y_s,\tilde \La_s),\d W_s\raa-\frac 12\int_{0}^{T}\!\!|\sigma^{-1}(Y_s)Z(Y_s,\tilde \La_s)|^2\d s\Big)\p
\end{equation}
is a new probability measure. Moreover, $(Y_t,\tilde \La_t)_{t\in[0,T]}$ under $\wt \q$ has the same distribution as that of $(\wt X_t,\tilde \La_t)$ under $\p$.

\begin{mylem}\label{lem-2.1}
Let $G:\R^d\times\S\ra \R_+$ be a measurable function and $\beta>0$ be a constant. Let $T>0$ be fixed.
\begin{itemize}
  \item[$\mathrm{(i)}$] If there exists a constant $\xi>d$ such that $\max_{i\in\S}\mu_0\big(G^\xi(\cdot,i)\big)<\infty$, then
      \begin{equation}\label{e-3.3}
      \E\Big[\int_0^TG(Y_s,\La_s)\d s\Big]\leq C\max_{i\in\S}\mu_0\big(G^\xi(\cdot,i)\big)^{\frac{1}{\xi}}<\infty
      \end{equation}
      for some constant $C=C(T,\xi,K_0)>0$.
  \item[$\mathrm{(ii)}$] If there exists a constant $\eta$ such that $\eta>\beta T d$ and $\max_{i\in\S}\mu_0\big(\e^{\eta G(\cdot,i)}\big)<\infty$,
      then
      \begin{equation}\label{e-3.4}
      \E\Big[\e^{\beta\int_0^TG(Y_s,\La_s)\d s}\Big]<\infty.
      \end{equation}
\end{itemize}
\end{mylem}

\begin{proof}
  We first prove $\mathrm{(ii)}$, then $\mathrm{(i)}$ follows easily from the derivation of $\mathrm{(ii)}$.
  Let $P_t^0$ denote the semigroup corresponding to the process $(Y(t))$ defined by \eqref{pro-y} with initial value $Y(0)=x$. Hence, the semigroup $P_t^0$ is symmetric w.r.t.\!\! $\mu_0$. Since $V $ satisfies condition (A), according to
  \cite[Theorem 1.1]{Wa11}, for $p>1$, the following Harnack inequality holds:
  \begin{equation}\label{3.2}
    \Big(P_t^0f(x)\Big)^p\leq P_t^0f^p(y)\exp\Big[\frac{K_0\sqrt{p}}{\sqrt{p}-1}\cdot\frac{|x-y|^2}{1-\e^{-K_0t}}\Big],\quad \forall\,f\in \mathscr B_b^+(\R^d).
  \end{equation}
  Applying the Harnack inequality \eqref{3.2} and the mutual independence between $(\La_t)$ and $(W_t)$, we get for any $\gamma>0$ and $K>0$
  \begin{align*}&\Big\{\E\Big[ \e^{\gamma G(Y_t,\La_t)\wedge K}\Big|\mathscr F_t^{\La}\Big]\Big\}^p=\Big\{P_t^0\e^{\gamma G(\cdot,\La_t)\wedge K}\Big\}^p(x)\\
  &\leq \Big\{P_t^0\e^{\gamma p G( \cdot,\La_t)\wedge K}\Big\}(y)\exp\Big[\frac{K_0\sqrt{p}}{\sqrt{p}-1}\cdot\frac{|x-y|^2}{1-\e^{-K_0t}}\Big].
  \end{align*}
  Passing to the limit as $K\ra +\infty$, it follows from Fatou's lemma that
  \begin{equation}\label{3.3}
  \Big\{P_t^0\e^{\gamma G(\cdot,\La_t)}\Big\}^p(x)\leq  \Big\{P_t^0\e^{\gamma pG(\cdot,\La_t)}\Big\}(y)\exp\Big[\frac{K_0\sqrt{p}}{\sqrt{p}-1}\cdot\frac{|x-y|^2}{1-\e^{-K_0t}}\Big].
  \end{equation}
  Denote $B(x,r)=\{y\in\R^d; |y-x|\leq r\}$ for $r>0$, $x\in\R^d$. Integrating both sides of \eqref{3.3} w.r.t. $\mu_0$ over the set $B(x, \sqrt{1-\e^{-K_0t}})$, we obtain
  \begin{equation}\label{3.4}
  \begin{split}
  &\Big\{P_t^0\e^{\gamma G(\cdot,\La_t)}(x)\Big\}^{p}\mu_0\big(B\big(x,\sqrt{1-\e^{-K_0t}}\big)\big)\\
  &\leq \int_{B\big(x,\sqrt{1-\e^{-K_0t}}\big)}\big\{P_t^0\e^{\gamma p G(\cdot,\La_t)}\big\}(y)\e^{\frac{K_0 \sqrt{p}}{\sqrt{p}-1} \cdot \frac{|x-y|^2}{1-\e^{-K_0t}}}\mu_0(\d y)\\
  &\leq \int_{B\big(x,\sqrt{1-\e^{-K_0t}}\big)}\big\{P_t^0\e^{\gamma p G(\cdot,\La_t)}\big\}(y)\e^{\frac{K_0\sqrt{p}}{\sqrt{p}-1} }\mu_0(\d y)\\
  &\leq \e^{\frac{K_0\sqrt{p}}{\sqrt{p}-1} }\mu_0(\e^{\gamma p G(\cdot,\La_t)}).
  \end{split}
  \end{equation}
  Since $\mu_0$ has strictly positive and continuous density $\e^{-V}$ w.r.t. the Lebesgue measure, there exists $\Gamma\in C(\R^d;(0,\infty))$ such that $\mu_0(B(x,t))\geq \Gamma(x) t^d$ for $t\in (0,1]$ and $x\in\R^d$. Invoking \eqref{3.4}, we obtain
  \begin{equation}\label{3.5}
  \E\e^{\gamma G(Y_t,\La_t)}\leq \Gamma(x)^{-\frac 1p}\e^{\frac{K_0}{p-\sqrt p}}\max_{i\in\S}\mu_0\Big(\e^{\gamma p G(\cdot,i)}\Big)^{\frac 1p}\frac{1}{\big(1-\e^{-K_0t}\big)^{  d/p}}, \ \quad t\in(0,T].
  \end{equation}
  Combining this with Jensen's inequality, one has
  \begin{equation}\label{3.6}
  \begin{aligned}
    \E\Big[\e^{\beta\int_0^TG( Y_t,\La_t)\d t}\Big]&\leq \frac 1T\int_0^T \E\Big[\e^{\beta TG( Y_t,\La_t)}\Big]\d t\\
    &\leq \frac{C}{\Gamma(x)^{1/p}}\max_{i\in\S}\mu_0\Big(\e^{\beta T pG( \cdot,i)}\Big)^{1/p}\int_0^T\frac{1}{(1-\e^{-K_0t})^{d/p}} \d t,
  \end{aligned}
  \end{equation} where $C=C(p,T,K_0)$ is a constant and $x$ is the initial value of $(Y_t)$.
  Taking $d<p<\frac{\eta}{\beta T }$ in \eqref{3.6},   it follows from the assumed condition in (ii) that
  \[\E\Big[\e^{\beta\int_0^TG( Y_t,\La_t)\d t}\Big]<\infty.\]
 In order to establish \eqref{e-3.3}, noticing $\xi>d$,  we obtain from \eqref{3.5} that
   \begin{equation}\label{n-1}
     \E[G( Y_t,\La_t)]\leq \frac{\e^{\frac{K_0}{\xi-\sqrt{\xi}} }\max_{i\in\S}\mu_0(G^{\xi}(\cdot,i))^{\frac 1\xi}}
     {\Gamma(x)^{\frac{1}{\xi}}(1-\e^{-K_0t})^{\frac{d}{\xi}}},\quad t\in (0,T],
   \end{equation}
   and hence
   \begin{align*}
     \E\Big[\int_0^T G( Y_t,\La_t)\d t\Big]\leq \frac{\e^{\frac{K_0}{\xi-\sqrt{\xi}} }}{\Gamma(x)^{\frac{1}{\xi}}}
     \Big(\int_0^T\!\!\frac{1}{(1-\e^{-K_0 t})^{ d/\xi }}\d t\Big)\max_{i\in\S}\mu_0(G^{\xi}( \cdot,i))^{\frac 1\xi}<\infty.
   \end{align*}
   The proof is complete.
\end{proof}

\noindent\textbf{Proof of Theorem \ref{t3} } For every Markov chain $(\La_t)$ with transition rate matrix $Q$, there is a unique strong solution to SDE \eqref{c-1} under the conditions imposed in this theorem, which of course implies the weak uniqueness of the solution to SDE \eqref{c-1}. Similarly, weak uniqueness holds for SDE \eqref{g-1}. In this proof, let $(\La_t)$ be the Markov chain given by \eqref{lam-1}, and $(\tilde \La_t)$ be given by \eqref{lam-2}. All the results established in beginning of this section still hold for this special construction of Markov chains. We shall this coupling process $(\La_t,\tilde \La_t)$ to estimate $\E\int_0^T\mathbf 1_{\{\La_s\neq \tilde \La_s\}}\d s$ using Lemma \ref{lem-2} in the following argument.

 By Lemma \ref{lem-2.1}, Novikov's conditions \eqref{Nov-1} and \eqref{Nov-2} are verified under the assumption of this theorem. Therefore, $(X_t,\La_t)_{t\in [0,T]}$ and $(\wt X_t,\tilde \La_t)_{t\in[0,T]}$ can be represented in terms of $(Y_t,\La_t)_{t\in [0,T]}$ and $(Y_t,\tilde \La_t)_{t\in [0,T]}$. Denote the initial value of $(Y_t)$ by $x_0$. It follows that for any measurable $f$ with $\|f\|_{\mathrm{Lip}}+\|f\|_\infty\leq 1$, and any $t\in[0,T]$,
\begin{equation}\label{f-1}
\begin{split}
  |\E f(X_t)-\E f(\wt X_t)|&=\big|\E_{\q} f(Y_t)-\E_{\wt \q}f(Y_t)\big|\\
                          &=\Big|\E\Big[\Big(\frac{\d \q}{\d \p}-\frac{\d \wt \q}{\d \p}\Big)f(Y_t)\Big]\Big| \leq \E\Big|\frac{\d \q}{\d \p}-\frac{\d \wt \q}{\d \p}\Big|.
\end{split}
\end{equation}
Setting
\[M_t=\int_0^t\la\sigma^{-1}(Y_s)Z(Y_s,\La_s),\d W_s\raa,\quad \wt M_t=\int_0^t\la\sigma^{-1}(Y_s)Z(Y_s,\tilde \La_s),\d W_s\raa,\]
and
\[\la M\raa_t=\int_0^t|\sigma^{-1}(Y_s)Z(Y_s,\La_s)|^2\d s,\quad \la \wt M\raa_t=\int_0^t|\sigma^{-1}(Y_s)Z(Y_s,\tilde \La_s)|^2\d s
\] for $t\in[0,T]$, by the inequality $|\e^x-\e^y|\leq (\e^x+\e^y)|x-y|$ for all $x,\,y\in\R$, we obtain that
\begin{equation}\label{f-2}
  \begin{split}
    &|\E f(X_t)-\E f(\wt X_t)|\\
    &\leq \E\Big[\Big(\frac{\d \q}{\d \p}+\frac{\d \wt \q}{\d \p}\Big)^p\Big]^{\frac 1p}\E\big[|M_T-\wt M_T-\frac 12 \la M\raa_T+\frac 12 \la \wt M\raa_T|^q\big]^{\frac 1q}
  \end{split}
\end{equation} for $p,\,q>1$ with $1/p+1/q=1$.

For the first term in \eqref{f-2}, since $\eta>2Td$, we can choose $p=p_0>1$ such that $q_0=p_0/(p_0-1)>2$ and $2p_0^2Td<\eta$.
\begin{align*}
  \E\Big[\Big(\frac{\d \q}{\d \p}\Big)^{p_0}\Big]&=\E\big[\exp\big(p_0M_T-\frac{p_0}{2}\la M\raa_T\big)\big]\\
  &\leq \E\big[\exp(2p_0M_T-2p_0^2\la M\raa_T)\big]^{\frac 12}\E\big[\exp(p_0(2p_0-1)\la M\raa_T)\big]^{\frac 12}.
\end{align*}
According to Lemma \ref{lem-1},
\[\E\big[\e^{2p_0^2\la M\raa_T}\big]<\infty,\quad \E\big[\e^{p_0(2p_0-1)\la M\raa_T}\big]<\infty.\]
Hence, $t\mapsto \exp\big(2p_0M_t-2p_0^2\la M\raa_t\big)$ is an exponential martingale for $t\in[0,T]$ and
\begin{equation}\label{f-3}
\E\Big[\Big(\frac{\d \q}{\d \p}\Big)^{p_0}\Big]\leq \frac{C}{\Gamma(x_0)^{\frac 1{  p_1}}}\max_{i\in\S}\mu_0\Big(\e^{\eta |\sigma^{-1}(\cdot)Z(\cdot,i)|^2}\Big)^{\frac{1}{  p_1}}\!\int_0^T\!\!\frac{1}{(1\!-\!\e^{-K_0t})^{\frac d{ p_1}}}\d t<\infty,
\end{equation} where $p_1>d$ satisfies $2p_0^2p_1T<\eta$, and  $C=C(p_1,T,K_0)$.


We proceed to estimate the second term in \eqref{f-2}. We shall estimate $\E[|M_T-\wt M_T|^{q_0}]$ and $\E[|\la M\raa_T-\la \wt M\raa_T|^{q_0}]$ separately. Since $q_0>2$, it follows from Burkholder-Davis-Gundy's inequality and Jensen's inequality that
\begin{align*}
  &\E[|M_T-\wt M_T|^{q_0}]\\
  &\leq C_{q_0}\E\Big[\Big(\int_0^T|\sigma^{-1}(Y_s)(Z(Y_s,\La_s)-Z(Y_s,\tilde\La_s))|^2\d s\Big)^{\frac{q_0}2}\Big]\\
  &\leq C_{q_0} T^{\frac {q_0}2-1}\E\Big[\int_0^T|\sigma^{-1}(Y_s)(Z(Y_s,\La_s)-Z(Y_s,\tilde\La_s))|^{q_0}\d s\Big]\\
  &=C_{q_0}T^{\frac {q_0}2-1}\E\Big[\int_0^T|\sigma^{-1}(Y_s)(Z(Y_s,\La_s)-Z(Y_s,\tilde\La_s))|^{q_0}
  \mathbf{1}_{\{\La_s\neq \tilde \La_s\}}\d s\Big]\\
  &\leq C_{q_0}T^{\frac {q_0}2-1}\int_0^T\E\Big[|\sigma^{-1}(Y_s)(Z(Y_s,\La_s)-Z(Y_s,\tilde\La_s))|^{2q_0}\Big]^\frac 12\p\big( \La_s\neq \tilde \La_s \big)^{\frac 12}\d s\\
  &\leq C_{q_0}T^{\frac{q_0}2\!-\!1}\Big(\!\int_0^T\!\!\!\E\Big[|\sigma^{-\!1}(Y_s)
    (Z(Y_s,\La_s)\!-\!Z(Y_s,\tilde\La_s))|^{2q_0}\Big]\d s\Big)^{\frac 12}  \Big(\!\int_0^T\!\!\!\p(\La_s\!\neq\! \tilde \La_s)\d s\Big)^{\frac 12}.
\end{align*}
By \eqref{n-1} of Lemma \ref{lem-1},
\begin{equation}\label{f-4}
\begin{aligned}
  &\E\Big[ \big|\sigma^{-1}(Y_s)(Z(Y_s,\La_s)-Z(Y_s,\tilde \La_s))\big|^{2q_0}\Big]\\
  &\leq 2^{2q_0-1}\Big(\E\Big[\big|\sigma^{-1}(Y_s)Z(Y_s,\La_s)\big|^{2q_0}\Big]
  +\E\Big[\big|\sigma^{-1}(Y_s)Z(Y_s,\La_s)\big|^{2q_0}\Big]\Big)\\
  &\leq 2^{2q_0}\e^{\frac{K_0}{\xi-\sqrt{\xi}}}\frac{\max_{i\in\S}
  \mu_0\Big(|\sigma^{-1}(\cdot)Z(\cdot,i)|^{2q_0\xi}\Big)^{\frac1{\xi}}}{\Gamma(x)(1-\e^{-K_0 s})^{\frac d{\xi}}}, \quad \xi>d.
\end{aligned}
\end{equation}
Note that the finiteness of $\displaystyle \max_{i\in\S}\mu_0\Big(|\sigma^{-1}(\cdot)Z(\cdot,i)|^{2q_0\xi}\Big)$ follows easily from the assumption
 \[ \max_{i\in\S}\mu_0\Big(\e^{\eta |\sigma^{-1}(\cdot)Z(\cdot,i)|^2}\Big)<\infty.\]
Therefore,
\begin{equation}\label{f-5}
\begin{split}
  &\E[|M_T\!-\!\wt M_T|^{q_0}]\\
  &\leq C\e^{\frac{K_0}{2\xi-2\sqrt{\xi}}}
  \Big(\!\int_0^T\! \frac{\max_{i\in\S}
  \mu_0\Big(|\sigma^{-1}(\cdot)Z(\cdot,i)|^{2q_0\xi}\Big)^{\frac 1{\xi}}}{(1-\e^{-K_0s})^{\frac{d}{\xi}}}\d s\Big)^{\frac 12}\Big(\!\int_0^T\!\!\p(\La_s\neq \tilde \La_s)\d s\Big)^{\frac 12}
\end{split}
\end{equation} for some constant $C=C(q_0,T)>0$ and $\xi>d$.
By Lemma \ref{lem-2}, we obtain that
\begin{equation}\label{f-6}
 \begin{split}
     & \E [|M_T\!-\!\wt M_T|^{q_0}]\\
      & \leq C\e^{\frac{K_0}{2\xi-2\sqrt{\xi}}}
  \Big(\!\int_0^T\! \frac{\max_{i\in\S}
  \mu_0\Big(|\sigma^{-1}(\cdot)Z(\cdot,i)|^{2q_0\xi}\Big)^{\frac 1{\xi}}}{(1-\e^{-K_0s})^{\frac{d}{\xi}}}\d s\Big)^{\frac 12}\cdot NT\|Q-\wt Q\|_{\ell_1}^{\frac 12}.
 \end{split}
\end{equation}

In the following,   we shall estimate $\E[|\la M\raa_T-\la \wt M\raa_T|^{q_0}]$.
\begin{align*}
  &\E\big[\big|\la M\raa_T-\la \wt M\raa_T\big|^{q_0}\big]\\
  &\leq \E\Big[\Big(\!\int_0^T\!\!|\sigma^{-1}(Y_s)(Z(Y_s,\La_s)\!-\! Z(Y_s,\tilde \La_s)|\big(|\sigma^{-1}(Y_s)Z(Y_s,\La_s)|\!+\!|\sigma^{-1}(Y_s)  Z(Y_s,\tilde \La_s|)\big) \d s\Big)^{q_0}\Big]\\
  &\leq \E\Big[\Big(\!\int_0^T\!\!|\sigma^{-1}(Y_s)(Z(Y_s,\La_s)\!-\! Z(Y_s,\tilde \La_s)|^{\gamma} \d s\Big)^{\frac{q_0}{\gamma}}\\
  &\qquad \quad \cdot\Big(\!\int_0^T\!\!\big(|\sigma^{-1}(Y_s)Z(Y_s,\La_s)|\!+\!|\sigma^{-1}(Y_s) Z(Y_s,\tilde \La_s)|\big)^{\gamma'} \d s\Big)^{\frac{q_0}{\gamma'}}\Big]\\
  &\leq \E\Big[\Big(\!\int_0^T\!\!|\sigma^{-1}(Y_s)(Z(Y_s,\La_s)\!-\! Z(Y_s,\tilde \La_s)|^{\gamma}\d s\Big)^{q_0}\Big]^{\frac1\gamma}\\
  &\qquad \quad \cdot\E\Big[\Big(\!\int_0^T \!\! \big(|\sigma^{-1}(Y_s)Z(Y_s,\La_s)|\!+\!|\sigma^{-1}(Y_s) Z(Y_s,\tilde Y_s)|\big)^{\gamma'} \d s\Big)^{q_0}\Big]^{\frac{1}{\gamma'}},
\end{align*} where $\gamma,\,\gamma'>1$ satisfy $1/\gamma+1/\gamma'=1$.
By Lemma \ref{lem-1}, it is easy to see
\[\E\Big[\Big(\!\int_0^T \!\! \big(|\sigma^{-1}(Y_s)Z(Y_s,\La_s)|\!+\!|\sigma^{-1}(Y_s) Z(Y_s,\tilde Y_s)|\big)^{\gamma'} \d s\Big)^{q_0}\Big]^{\frac{1}{\gamma'}}<\infty.\]
On the other hand,
 \begin{align*}
   &\E\Big[\Big(\!\int_0^T|\sigma^{-1}(Y_s)(Z(Y_s,\La_s)-Z(Y_s,\tilde \La_s))\d s\Big)^{q_0}\Big]\\
   &\leq T^{q_0-1}\Big(\int_0^T\!\E\big[|\sigma^{-1}(Y_s)(Z(Y_s,\La_s)-Z(Y_s,\tilde \La_s))|^{2\gamma q_0}\big]\d s\Big)^{\frac 12}\Big(\int_0^T\!\p(\La_s\neq\tilde \La_s)\d s\Big)^{\frac 12}.
 \end{align*}
 By virtue of \eqref{f-4} and Lemma \ref{lem-2}, we get
 \begin{equation}\label{f-7}
   \begin{split}
       & \E\Big[\Big(\!\int_0^T|\sigma^{-1}(Y_s)(Z(Y_s,\La_s)-Z(Y_s,\tilde \La_s))\d s\Big)^{q_0}\Big]^{\frac 1\gamma} \\
        & \leq C\e^{\frac{K_0}{2\gamma(\xi-\sqrt{\xi})}}
  \Big(\!\int_0^T\! \frac{\max\limits_{i\in\S}
  \mu_0\Big(|\sigma^{-1}(\cdot)Z(\cdot,i)|^{2q_0\gamma\xi}\Big)^{\frac 1{\xi}}}{(1-\e^{-K_0s})^{\frac{d}{\xi}}}\d s\Big)^{\frac 1{2\gamma}} N^{\frac 1{\gamma}}\|Q-\wt Q\|_{\ell_1}^{\frac 1{2\gamma}},
   \end{split}
 \end{equation} where $C=C(T,x_0,q_0)$ is a positive constant.

 In all, inserting the estimates \eqref{f-3}, \eqref{f-6} and \eqref{f-7} into \eqref{f-2},
 we arrive at
 \begin{align*}
   \big|\E f(X_t)-\E f(\wt X_t)\big|&\leq C \big(\|Q-\wt Q\|_{\ell_1}^{\frac 1{2q_0}}\vee\|Q-\wt Q\|_{\ell_1}^{\frac{1}{2q_0\gamma}}\big)
 \end{align*} for some constant $C$  depending on $N, T, x_0, \tau_1, K_0,\xi,\gamma, p_0, \max_{i\in\S}\mu_0\big(\e^{\eta|\sigma^{-1}(\cdot)Z(\cdot,i)|^2}\big)$,  and $\gamma>1$.
 By virtue of the definition of  $W_{bL}(\cdot,\cdot)$,
 \[W_{bL}(\mathcal{L}(X_t),\mathcal{L}(\wt X_t))\leq C \big(\|Q-\wt Q\|_{\ell_1}^{\frac 1{2q_0}}\vee\|Q-\wt Q\|_{\ell_1}^{\frac{1}{2q_0\gamma}}\big).
 \]
 This completes the proof.

 \noindent\textbf{Proof of Theorem \ref{t4}}  This theorem can be proved along the same line as Thereom \ref{t3} by noting $\|Q-\wt Q\|_{\ell_1}\leq \|B\|_{\ell_1}+\|Q_1+\hat Q\|_{\ell_1}$. The details are omitted.

\section{Further discussion}

Recall the expression \eqref{ext} of $Q$. The probabilistic meaning of $q_0$ is that the Markov chain $(\La_t)$ stays at the state ``0" for a random period distributed as an exponential distribution with parameter $q_0$. So the larger the value of $q_0$ is, the shorter time period the process $(\La_t)$ will stay at ``0" in average. One may consider  a limitation case that  $q_0$ equals to $+\infty$, that is,
\[Q_\infty=\begin{pmatrix}
  -\infty&\alpha\\
  \beta& Q_1
\end{pmatrix},\]
which means that the jump will occur immediately once the process $(\La_t)$ reaches the state ``0". The state ``0" in $Q_\infty$ is called  an instantaneous state. It seems also interesting to  study the asymptotic behavior of $Q$ to $Q_\infty$ as $q_0$ tends to $+\infty$. Note that the continuous time Markov chain with instantaneous state produces new phenomenon compared with the Markov chains which are totally stable. For example, consider the well-known example provided by Kolmogorov \cite{Kol}:
\begin{equation*}
  Q=\begin{pmatrix}
    -\infty & 1&1&1&\ldots\\
    q_1&-q_1&0&0&\ldots\\
    q_2&0&-q_2&0&\ldots\\
    q_3&0&0&-q_3&\ldots\\
    \ldots&\ldots &\ldots&\ldots&\ldots
  \end{pmatrix}
\end{equation*}
It was shown by Kendall and Reuter \cite{KR} that if
\[\sum_{j=1}^\infty q_j^{-1}<+\infty,\]
then there exists a Markov process with the generator $Q$. Notice that the state space of this Markov process is denumerable. Moreover, Chen and Reushaw \cite{CR} presented some sufficient conditions for the existence and uniqueness of continuous-time Markov chains with instantaneous states. According to
\cite[Corollary 3.2]{CR},  Markov chains with a finite  states have no instantaneous states.  In the present work  the state space $\S$ of Markov chain is  finite,  we have not consider that the  Markov chain has the generator  $Q_\infty$, and hence the  corresponding processes $(\La_t)$ and $(X_t)$ have not been discussed. Therefore, to study the current problems for regime-switching processes with infinite state space $\S$ and instantaneous state is meaningful,  and we leave it  for further investigation.

\noindent\textbf{Acknowledgement}  We are grateful to Professor Yonghua Mao for sharing his unpublished results,  and valuable discussion on the works of \cite{Mit03,Mit04}.


\begin{thebibliography}{17}


\bibitem{BS} J. Bao, J. Shao, Permanence and extinction of regime-switching predator-prey models, SIAM J. Math. Anal., 48 (2016), 725-739.

\bibitem{Bar} J. Bardet,  H. Guerin, F. Malrieu, Long time behavior of diffusions with Markov switching, ALEA Lat. Am. J. Probab. Math. Stat., 7 (2010), 151-170.

\bibitem{BBG}  {G.  Basak, A. Bisi,   M.  Ghosh}, {  Stability of a random diffusion with linear drift},   {  J. Math. Anal. Appl.}, { 202} (1996), 604-622.

\bibitem{BBG99}{G.  Basak, A. Bisi, M.  Ghosh,
Stability of a degenerate diffusions with state-dependent switching,
{ J. Math. Anal. Appl.}, 240 (1999), 219-248.}

\bibitem{BD} S.  Brown,   P.  Dybvig, The empirical implications of the Cox, Ingersoll, Ross theory of the term structure of interest rates, Journal of Finance, 41 (1986), 617-630.

\bibitem{CR} A.  Chen, E. Renshaw, Existence and uniqueness criteria for conservative
uni-instantaneous denumerable Markov processes, Probab. Theory Relat. Fields, 94 (1993), 427-456.

\bibitem{Chen} M.-F. Chen, From Markov chains to non-equilibrium particle systems, 2nd ed. Singapore: World Scientific, 2004.

\bibitem{CH} {  B. Cloez,   M. Hairer}, {  Exponential ergodicity for Markov processes with random switching}, {  Bernoulli},  {  21}  (2015),    505-536.

\bibitem{DY} de Saporta, J.  Yao,  Tail of a linear diffusion with Markov switching, Ann. Appl. Probab., 15 (2005), (1B), 992-1018.

\bibitem{GAM} M. Ghosh, A. Arapostathis, S. Marcus, Optimal control of switching diffusions with application to flexible manufacturing systems, {SIAM J. Contr. Optim.},  30 (1992), 1-23.

\bibitem{HS18} T. Hou, J. Shao, Heavy tail and light tail of Cox-Ingersoll-Ross processes  with regime-switching, to appear in Sci. China Math. 2019 or arXiv:1709.01691.

\bibitem{KR} D.  Kendall, G. Reuter, Some pathological Markov processes with a denumerable infinity of states and the associated semigroups of operators on $\ell$, Proc. Intern. Congr. Math. Amsterdam, Vol. III, 377-415. Amserdam: North-Holland 1954.

\bibitem{Kol} A.  Kolmogorov, On the differentiability of the transition probabilities in homogeneous Markov processes with a denumerable number of states, Moskov. Gos. Univ. Ucenye Zapiski MGY 148 Mat. 4 (1951), 53-59.

\bibitem{M99}X.  Mao, Stability of stochastic differential equations with Markovian switching,
Stoch. Process. Appl., 79 (1999), 45-67.

\bibitem{MY} X.  Mao, C.  Yuan,  Stochastic Differential Equations with Markovian Switching, Imperial College Press, London, 2006.

\bibitem{Mit03} A.  Mitrophanov, Stability and exponential convergence of continuous-time Markov chains, J. Appl. Prob., 40 (2003), 970-979.

\bibitem{Mit04} A.  Mitrophanov, The spectral gap and perturbation bounds for reversible continuous-time Markov chains, J. Appl. Prob., 41 (2004), 1219-1222.


\bibitem{PP} {M. Pinsky,   R. Pinsky},  {Transience recurrence and central limit theorem behavior for diffusions in random temporal environments}, Ann. Probab., { 21} (1993), 433--452.

\bibitem{Sh1} {  J. Shao},  {Ergodicity of one-dimensional regime-switching diffusion processes}, {  Science China Math.}, {57} (2014), 2407-2414.

\bibitem{Sh14b} {J. Shao},  { Criteria for transience and recurrence of regime-switching diffusion processes},   Electron. J. Probab., 20 (2015),  1-15.

\bibitem{Sh15} {J. Shao},  Strong solutions and strong Feller properties for regime-switching diffusion processes in an infinite state space, SIAM J. Control Optim., 53 (2015), 2462-2479.

\bibitem{Sh18} J. Shao, Invariant measures and Euler-Maruyama's approximations of state-dependent regime-switching diffusions. SIAM J. Control Optim., 56 (2018), no. 5, 3215-3238.

\bibitem{SX2} {J. Shao,   F.  Xi},  {Stability and recurrence of regime-switching diffusion processes}, SIAM J. Control Optim.,  52 (2014), 3496-3516.

\bibitem{Sk89}{A. Skorokhod, Asymptotic Methods in the Theory of Stochastic
Differential Equations, American Mathematical Society, Providence,
RI. 1989.}

\bibitem{Vi} C. Villani, Topics in optimal transportation, American Mathematical Society, Providence, RI, 2003.

\bibitem{Vi09a} C. Villani, Optimal transport, old and new, Grundlehren der mathematischen
Wissenschaften, vol. 338, Springer Berlin Heidelberg, 2009.

\bibitem{Wa11} F.Y.  Wang, Harnack inequality for SDE with multiplicative noise and extension to Neumann semigroup on nonconvex manifolds, Ann. Probab., 39 (2011), 1149-1467.

\bibitem{Wa17} F.Y.  Wang, Integrability conditions for SDEs and semilinear SPDEs, Ann. Probab., 45 (2017), 3223-3265.

\bibitem{XY} F. Xi, G. Yin, Stability of regime-switching jump
diffusions, {  SIAM J. Control Optim.}, {48} (2010), 525-4549.

\bibitem{YZ}{G. Yin, C. Zhu, Hybrid switching diffusions: properties and applications, Vol. 63, Stochastic Modeling and Applied Probability, Springer, New York. 2010.}

\bibitem{ZI} A.I. Zeifman, D.L. Isaacson, On strong ergodicity for nonhomogeneous continuous-time Markov chains, Stochastic Process. Appl., 50 (1994), 263-273.

\bibitem{Zh16} X.C. Zhang, Stochastic differential equations with Sobolev diffusion and singular drift and applications, Annals Appl. Probab., 26 (2016), 2697-2732.

\end{thebibliography}
\end{document}